\documentclass[10pt]{amsart}

\usepackage{amsmath}
\usepackage{amssymb}
\usepackage{xcolor}

\usepackage{hyperref}
\definecolor{MITred}{RGB}{163,31,52}
\hypersetup{
    colorlinks=true,
    linkcolor=MITred,
    citecolor=MITred,
    urlcolor=MITred
}
\usepackage{url}
\usepackage{fancyhdr}
\usepackage{mathtools}
\usepackage{tikz-cd}
\usepackage{stmaryrd}
\usepackage{calrsfs}

\newtheorem{theorem}{Theorem}[section]
\newtheorem{lemma}[theorem]{Lemma}
\newtheorem{prop}[theorem]{Proposition}
\newtheorem{corollary}[theorem]{Corollary}

\theoremstyle{definition}
\newtheorem{definition}[theorem]{Definition}
\newtheorem{example}[theorem]{Example}

\theoremstyle{remark}

\numberwithin{equation}{section}

\newcommand{\aaa}{\mathbb{A}}
\newcommand{\cc}{\mathbb{C}}
\newcommand{\ff}{\mathbb{F}}
\newcommand{\nn}{\mathbb{N}}
\newcommand{\pp}{\mathbb{P}}

\newcommand{\qq}{\mathbb{Q}}
\newcommand{\rr}{\mathbb{R}}
\newcommand{\zz}{\mathbb{Z}}

\DeclareMathOperator{\mcd}{mcd}
\providecommand\ldb{\llbracket}
\providecommand\rdb{\rrbracket}

\newcommand{\gp}{\mathcal{G}}

\DeclareMathOperator{\ord}{ord}

\newcommand{\uu}{\mathcal{U}}

\newcommand{\fl}[1]{\lfloor#1\rfloor}

\DeclareMathOperator{\rank}{rank}
\DeclareMathOperator{\supp}{supp}



\begin{document}

\title{Maximal common divisors in monoid algebras}

\author{Grant Blitz}
\address{PRIMES-USA\\MIT\\Cambridge, MA 02139}
\email{grantstar16@gmail.com}

\author{Felix Gotti}
\address{Department of Mathematics\\MIT\\Cambridge, MA 02139}
\email{fgotti@mit.edu}

\author{Darren Han}
\address{Department of Mathematics\\MIT\\Cambridge, MA 02139}
\email{drh@mit.edu}

\author{Hengrui Liang}
\address{PRIMES-USA\\MIT\\Cambridge, MA 02139}
\email{spaceblastxy1@gmail.com}

\date{\today}

\begin{abstract}
    A commutative cancellative monoid has the MCD property if every nonempty finite subset has a maximal common divisor, and it has the MCD-finite property if every such subset has only finitely many maximal common divisors up to associates. We study these properties in monoid algebras. We first identify atomic and antimatter families of rank-$1$ torsion-free MCD monoids. We then prove that, for a field $F$ and a torsion-free monoid $M$, the MCD property ascends from $M$ to $F[M]$ whenever the group algebra $F[\gp(M)]$ satisfies the ascending chain condition on principal ideals. We also record the pre-Schreier case as a consequence of classical GCD ascent. In contrast, for every prime $p$, we construct a reduced rank-$3$ MCD monoid $M$ such that $\ff_p[M]$ is not an MCD domain, and an atomic rank-$1$ MCD monoid whose algebra over $\ff_p$ is not atomic. We further prove that the prime reciprocal monoids considered here are MCD-finite and that MCD-finiteness ascends to polynomial extensions. Finally, for every prime $p$, we construct a rank-$1$ torsion-free q-GCD monoid whose algebra over $\ff_p$ is not q-GCD, where the q-GCD property requires at most one maximal common divisor up to associates for every nonempty finite subset.
\end{abstract}

\bigskip
\maketitle

\bigskip
\section{Introduction}
\label{sec:intro}

The ascent of divisibility and factorization properties to polynomial rings and monoid algebras is a central problem in commutative algebra. For an integral domain $R$ and a commutative cancellative monoid $M$, the monoid algebra $R[M]$ consists of finite formal sums with coefficients in $R$ and exponents in $M$; in particular, $R[\nn_0]$ is naturally isomorphic to $R[x]$. Classical work of Gilmer and Parker~\cite{GP74} and Anderson and Anderson~\cite{AA82} relates divisibility in these algebras to divisibility in the coefficient domain and the exponent monoid. Gilmer's monograph~\cite{rG84} develops this theory systematically. We study two weaker forms of the greatest-common-divisor property: the existence and finiteness of maximal common divisors.
\smallskip

Let $M$ be a multiplicatively written commutative cancellative monoid. A common divisor $d$ of a nonempty finite subset $S$ of $M$ is a \emph{maximal common divisor} (MCD) if the only common divisors of $S/d:=\{s/d:s\in S\}$ are units. We call $M$ an \emph{MCD monoid} if every nonempty finite subset has an MCD. Every greatest common divisor (GCD) is an MCD, but distinct MCDs need not be associates. Thus, MCD monoids generalize GCD monoids, which include unique factorization monoids. GCD domains and related divisibility properties have a substantial literature; see~\cite{dA00,iK70,LP20}. An integral domain has any of these monoidal properties when its multiplicative monoid of nonzero elements does. Every monoid satisfying the ascending chain condition on principal ideals (ACCP) is an MCD monoid.
\smallskip

Roitman~\cite{mR93} studied the MCD property in connection with atomicity in polynomial extensions. The ascent problem for atomicity was raised by Gilmer~\cite[page~189]{rG84} for monoid algebras and subsequently by Anderson, Anderson, and Zafrullah~\cite[Question~2]{AAZ90} for polynomial rings. Roitman constructed an atomic domain whose polynomial ring is not atomic, while proving that atomicity does ascend from an atomic MCD domain~\cite[Proposition~1.1]{mR93}. More recently, the MCD property has been studied in power monoids~\cite{DGHLS25}, Puiseux monoids~\cite{LWZ24}, and semidomains~\cite{GP23}. Gotti and Rabinovitz~\cite{GR25} constructed an atomic Puiseux monoid whose monoid algebra over every integral domain is nonatomic. In the present paper, we investigate how imposing the MCD property changes these ascent questions.
\smallskip

In Section~\ref{sec:mcd}, we first identify atomic and antimatter families of rank-$1$ torsion-free MCD monoids. The atomic family extends the prime reciprocal monoids studied in~\cite{GL22,LWZ24}; the antimatter family consists of sums of nonnegative cones of groups $\zz[1/p]$. We then prove that, for a field $F$ and a torsion-free monoid $M$, the MCD property ascends to $F[M]$ when $F[\gp(M)]$ satisfies the ACCP. The classical characterization of factorial group algebras~\cite[Propositions~2.1 and 2.2]{AA82} expresses this hypothesis as the ascending chain condition on cyclic subgroups of $\gp(M)$; in particular, it holds when $\gp(M)$ is free abelian. We also explain why ascent for pre-Schreier MCD monoids follows from the classical GCD ascent theorem~\cite[Theorem~14.5]{rG84}.
\smallskip

The MCD property does not ascend to monoid algebras in general. For each prime $p$, we construct a reduced rank-$3$ MCD monoid $M$ for which $\ff_p[M]$ is not an MCD domain. Its monoid algebra contains a pair of elements whose common divisors form an increasing family with no maximal member. We also show that Roitman's positive result for atomic MCD domains does not extend to arbitrary monoid algebras: for every prime $p$, there is an atomic rank-$1$ torsion-free MCD monoid $M$ such that $\ff_p[M]$ is not atomic. This construction builds on~\cite{CG19,GGP25} and establishes the MCD property in addition to atomicity.
\smallskip

A monoid is \emph{MCD-finite} if every nonempty finite subset has only finitely many MCDs up to associates; the number of MCDs is allowed to be zero. Thus, MCD-finiteness and the MCD property impose different requirements. Eftekhari and Khorsandi~\cite{EK18} introduced MCD-finiteness and proved that a polynomial ring $R[x]$ is IDF if and only if $R$ is both IDF and MCD-finite. Here, IDF means that each nonzero element has only finitely many irreducible divisors up to associates. Their theorem characterizes the additional condition needed for IDF ascent, following the negative answer of Malcolmson and Okoh~\cite{MO09} to the question in~\cite{AAZ90}. Further developments concerning IDF and MCD-finiteness appear in~\cite{GZ23,GP25,DG26}.
\smallskip

In Section~\ref{sec:mcd-finite}, we prove that the prime reciprocal monoids introduced in Section~\ref{sec:mcd} are MCD-finite, although they do not satisfy the ACCP and hence are not finite factorization monoids. We then prove that MCD-finiteness ascends to polynomial extensions, using a coefficient-normalization argument related to~\cite[Theorem~2.1]{EK18}. Finally, we consider the \emph{q-GCD property}, which requires at most one MCD up to associates for each nonempty finite subset. Pre-Schreier monoids have this property, but q-GCD monoids need not have MCDs. For every prime $p$, we construct a rank-$1$ torsion-free q-GCD monoid $M$ such that $\ff_p[M]$ has two nonassociate MCDs of the same pair of elements.

\bigskip
\section{Background}
\label{sec:background}

We proceed to introduce most of the terminology, notation, and known results we shall be using later throughout the paper.

\medskip
\subsection{General Notation}

As is customary, $\zz$, $\qq$, $\rr$, $\aaa$, and $\cc$ will denote the sets of integers, rational numbers, real numbers, algebraic complex numbers, and complex numbers, respectively. We let $\nn$ and $\nn_0$ denote the sets of positive and nonnegative integers, respectively. In addition, we let $\pp$ denote the set of primes. For $p \in \pp$ and $n \in \nn$, we let $\ff_{p^n}$ be the finite field of cardinality $p^n$. For $a,b \in \zz$, we let $\ldb a,b \rdb$ denote the following set of integers, which is empty when $a>b$:
\[
    \ldb a,b \rdb := \{n \in \zz : a \le n \le b\}.
\]
In addition, for $S \subseteq \rr$ and $r \in \rr$, we set $S_{\ge r} := \{s \in S : s \ge r\}$ and $S_{> r} := \{s \in S : s > r\}$. For $r \in \qq$, let $\mathsf{d}(r)$ denote its positive denominator in lowest terms, with $\mathsf{d}(0):=1$. For each prime $p$, let $v_p$ denote the $p$-adic valuation on $\qq$, with $v_p(0):=\infty$.

\medskip
\subsection{Commutative Monoids}

A multiplicatively written commutative semigroup $S$ is called \emph{cancellative} if for all $a,b,c \in S$ the equality $ab = ac$ implies that $b=c$. Although a monoid is usually defined to be a semigroup with an identity element, for the purposes of this paper the term \emph{monoid} refers to a cancellative and commutative semigroup with an identity element.
\smallskip

Let $M$ be a multiplicatively written monoid. We let $M^\bullet := M \setminus \{1\}$; when $M$ is written additively, we use $M^\bullet := M \setminus \{0\}$. The \emph{group of units} of~$M$ is the abelian group consisting of all invertible elements of~$M$, which is denoted by either $M^\times$ or $\uu(M)$. We say that $M$ is \emph{reduced} if the group of units of $M$ is trivial. The quotient monoid $M/M^\times$ is a reduced monoid, which is denoted by $M_{\text{red}}$ and called the \emph{reduced monoid} of $M$. If $M$ is reduced then we can naturally identify each element $a \in M$ with its class $aM^\times$ in $M_{\text{red}}$.
\smallskip

The cancellative monoid $M$ embeds in its \emph{Grothendieck group} $\gp(M)$, also called its \emph{quotient group}. In multiplicative notation, every element of $\gp(M)$ is a quotient of two elements of $M$; in additive notation, every element is a difference. The monoid $M$ is \emph{torsion-free} if $a^n=b^n$ implies $a=b$ for all $a,b\in M$ and $n\in\nn$, equivalently, if $\gp(M)$ is torsion-free. Its \emph{rank} is $\dim_{\qq}(\qq\otimes_{\zz}\gp(M))$. A torsion-free monoid has rank $1$ if and only if it is isomorphic to a nontrivial additive submonoid of $\qq$ (see~\cite[Section~24]{lF70} and~\cite[Theorem~2.9]{rG84}). Additive submonoids of $\qq_{\ge0}$ are called \emph{Puiseux monoids}. Their atomic structure and factorization theory are studied in~\cite{CGGP26}, and their monoid algebras are studied in~\cite{BGLZ24,GG25,GGP25,GR25}.

\medskip
\subsection{Divisibility}

For $a,b \in M$, we say that $b$ \emph{divides} $a$ and write $b \mid_M a$ if there exists $c \in M$ such that $a = bc$. For $b,c \in M$, if both relations $b \mid_M c$ and $c \mid_M b$ hold then $b$ and $c$ are called \emph{associates}. Let $S$ be a nonempty subset of $M$. An element $d \in M$ is called a \emph{common divisor} of $S$ if $d \mid_M s$ for all $s \in S$. A common divisor $d \in M$ is called a \emph{greatest common divisor} (GCD) of $S$ if every other common divisor of $S$ divides $d$: we let $\gcd_M(S)$ denote the set of all GCDs of $S$, dropping the subscript $M$ when we see no risk of ambiguity (any two GCDs of $S$ are associates). We call $M$ a \emph{GCD monoid} if every nonempty finite subset of $M$ has a GCD.
\smallskip

A common divisor $d \in M$ of $S$ is called a \emph{maximal common divisor} (MCD) of $S$ if the only common divisors of the set $S/d := \{s/d : s \in S\}$ are the units of~$M$. We let $\mcd_M(S)$ denote the set consisting of all the MCDs of $S$ in $M$, dropping the subscript $M$ when we see no risk of ambiguity. If every nonempty finite subset of $M$ has an MCD, we say that $M$ is an \emph{MCD monoid}. Following~\cite{EK18}, we say that $M$ is an \emph{MCD-finite monoid} if every nonempty finite subset of~$M$ has only finitely many MCDs up to associates. When a monoid is MCD-finite, we also say that it has the \emph{MCD-finite property}. If every nonempty finite subset of $M$ has at most one MCD up to associates then we say that $M$ is a \emph{q-GCD monoid} (or a \emph{quasi-GCD monoid}). The MCD-finite property permits a finite subset to have no MCDs. In particular, every q-GCD monoid is MCD-finite, but MCD-finiteness alone does not assert the MCD property.
\smallskip

The monoid $M$ is called a \emph{pre-Schreier monoid} if each element $a \in M$ is \emph{primal}, which means that, for all $b,c \in M$ such that $a \mid_M bc$, we can write $a = b' c'$ for some elements $b', c' \in M$ such that $b' \mid_M b$ and $c' \mid_M c$. Clearly, every unit is primal. If a monoid is pre-Schreier then we say that it has the \emph{pre-Schreier property}. A monoid has the \emph{GL property} if for every $a \in M \setminus M^\times$ the following condition holds: if $a \mid_M bc$ for some $b,c \in M$ then there exists $d \in M \setminus M^\times$ with $d \mid_M a$ such that $d \mid_M b$ or $d \mid_M c$.

\medskip
\subsection{Atomicity and Factorizations}

An element $a \in M \setminus M^\times$ is called an \emph{atom} if for all $b,c \in M$, the equality $a = bc$ implies that $M^\times \cap \{b,c\}$ is nonempty. We let $\mathcal{A}(M)$ denote the set consisting of all the atoms of $M$. An element $b \in M$ is called \emph{atomic} if either $b$ is a unit or $b$ factors into atoms (allowing repetitions). Following Cohn~\cite{pC68}, we say that the monoid~$M$ is \emph{atomic} if every element of~$M$ is atomic. In addition, $M$ is called \emph{strongly atomic} if $M$ is atomic and, for any $b,c \in M$, the subset $\{b,c\}$ of $M$ has an MCD.
\smallskip

An element $p \in M \setminus M^\times$ is called a \emph{prime} if for all $b,c \in M$ such that $p \mid_M bc$, either $p \mid_M b$ or $p \mid_M c$. As $M$ is assumed to be cancellative, every prime element of $M$ is an atom. A monoid is called an \emph{AP monoid} if every atom is prime. Following Coykendall, Dobbs, and Mullins~\cite{CDM99}, we say that $M$ is \emph{antimatter} if $\mathcal{A}(M)$ is the empty set. Observe that every antimatter monoid is trivially an AP monoid. Following Grams and Warner~\cite{GW75}, we say that~$M$ is an \emph{irreducible-divisor-finite monoid} (or an \emph{IDF monoid}) if every element of $M$ has only finitely many atom divisors up to associates.
\smallskip

For an atomic monoid $M$, let $\mathsf{Z}(M)$ denote the free commutative monoid on $\mathcal{A}(M_{\mathrm{red}})$, and let $\pi\colon\mathsf{Z}(M)\to M_{\mathrm{red}}$ be the unique monoid homomorphism that fixes every atom of $M_{\mathrm{red}}$. For each $b\in M$, we set
\[
    \mathsf{Z}(b) = \mathsf{Z}_M(b) = \pi^{-1} (bM^\times).
\]
The elements of $\mathsf{Z}_M(b)$ are the \emph{factorizations} of $b$, with order and associates disregarded. An atomic monoid $M$ is a \emph{finite factorization monoid} (FFM) if $\mathsf{Z}_M(b)$ is finite for every $b\in M$, and it is a \emph{unique factorization monoid} (UFM) if each such set is a singleton. A monoid is a UFM if and only if it is an atomic GCD monoid~\cite{GH06}. Every FFM is MCD-finite, since each of its elements has only finitely many divisors up to associates.
\smallskip

A monoid $M$ satisfies the \emph{ascending chain condition on principal ideals} (ACCP) if every chain $a_1M\subseteq a_2M\subseteq\cdots$ stabilizes; in additive notation, principal ideals have the form $a+M$. Every FFM satisfies the ACCP, and every ACCP monoid is atomic and has the MCD property~\cite{GH06}. For arbitrary monoids, the implications in Figure~\ref{fig:generalized-gcd} relate the divisibility properties defined above.
\smallskip
\begin{figure}[ht]
    \begin{tikzcd} [cramped]
        \textbf{ GCD } \arrow[r, Rightarrow] & \textbf{ pre-Schreier } \arrow[r, Rightarrow] &
        \textbf{ GL } \arrow[r, Rightarrow] &\textbf{AP}
    \end{tikzcd}
    \caption{Three properties that generalize the GCD property.}
    \label{fig:generalized-gcd}
\end{figure}

\medskip
\subsection{Integral Domains}

A nonzero commutative ring with identity is called an \emph{integral domain} if the product of any two nonzero elements is nonzero.
\smallskip

Let $R$ be an integral domain. We let $R^\times$ denote the group of units of $R$, while we let $R^*$ denote the multiplicative monoid consisting of all nonzero elements of~$R$. Observe that the nonzero principal ideals of $R$ can be obtained by adjoining $\{0\}$ to the principal ideals of the monoid $R^*$, whence we see that $R$ satisfies the ACC on principal ideals if and only if $R^*$ satisfies the ACCP, in which case we say that $R$ satisfies the \emph{ACCP}. Notice also that an integral domain $R$ is a UFD if and only if $R^*$ is a UFM.
\smallskip

The atoms and prime elements of $R$ are exactly those of the monoid $R^*$. For any $r,s \in R^*$ with $r \mid_{R^*} s$ we write $r \mid_R s$. In general, we can adapt every monoidal property to the setting of integral domains through their multiplicative monoids. For instance, we say that an integral domain $R$ is an \emph{MCD domain} (resp., \emph{MCD-finite domain}) provided that $R^*$ is an MCD monoid (resp., MCD-finite monoid). We can do the same for the rest of the monoidal properties we have previously defined to define the notions of a q-GCD domain, a pre-Schreier domain, a GL domain, an AP domain, an atomic domain, a strongly atomic domain, and an IDF domain.

\medskip
\subsection{Monoid Algebras}

For the rest of this subsection, let $R$ be an integral domain, and let $M$ be a monoid, which is written additively. For an indeterminate $x$, the commutative ring consisting of all polynomial expressions in~$x$ with exponents in $M$ and coefficients in $R$ is called the \emph{monoid algebra} of~$M$ over~$R$. Following the notation in~\cite{rG84}, we will denote the monoid algebra of~$M$ over $R$ by $R[x;M]$, although we simply write $R[M]$ if we see no danger of ambiguity. If $M$ is torsion-free (in addition to being cancellative and commutative) then $R[M]$ is an integral domain with
\[
    R[M]^\times = \{ux^m : u \in R^\times \text{ and } m \in M^\times \}
\]
(see~\cite[Theorems~8.1 and 11.1]{rG84}). We assume, for the rest of this subsection, that $M$ is a torsion-free monoid and, therefore, that $R[M]$ is an integral domain.
\smallskip

We say that $M$ is a \emph{linearly ordered monoid} with respect to a given total order relation $\preceq$ if for all $b,c,d \in M$ the inequality $b \prec c$ implies that $b+d \prec c+d$. It follows from Levi's theorem~\cite{fL13} that a commutative semigroup $S$ with an identity element can be turned into a linearly ordered monoid if and only if $S$ is both cancellative and torsion-free. Thus, as we are assuming that $M$ is torsion-free, we can assume the existence of a total order $\preceq$ on~$M$ that is compatible with its operation.
\smallskip

Let $M$ be a linearly ordered monoid under a total order relation, which we denote by~$\preceq$. Then we can write any nonzero polynomial expression $f$ in the monoid algebra $R[M]$ as follows:
\[
    f := c_1 x^{m_1} + \dots + c_k x^{m_k}
\]
for some nonzero $c_1, \dots, c_k \in R$ and exponents $m_1, \dots, m_k \in M$ with $m_1 \succ \dots \succ m_k$. In this case, we call $\supp \, f := \{m_1, \dots, m_k\}$ the \emph{support} of $f$, while we refer to $\deg f := m_1$ and $\ord \, f := m_k$ as the \emph{degree} and \emph{order} of $f$, respectively. Notice that when the monoid $M$ is $\nn_0$ under the usual order, we recover the standard notions of degree, order, and support for standard polynomials in~$R[x]$.
\smallskip

The monoid algebra $R[M]$ is a GCD domain if and only if $R$ is a GCD domain and $M$ is a GCD monoid~\cite[Theorem~14.5]{rG84}. In light of \cite[Corollary~2]{jO88}, when $R$ is Noetherian, the following dimension formula holds:
\[
    \dim R[M] = \dim R + \rank M.
\]
In particular, for a field $F$ and a rank-$1$ torsion-free monoid $M$, the monoid algebra $F[M]$ has Krull dimension $1$. Such algebras provide the setting for several constructions in this paper.
\smallskip

Background information on monoid algebras with emphasis on the ascent of algebraic properties from $(M,R)$ to $R[M]$, including the most significant progress on the same subject until 1984, can be found in Gilmer's book~\cite{rG84}.

\bigskip
\section{The MCD Property}
\label{sec:mcd}

We study the existence of MCDs, beginning with two families of rank-$1$ torsion-free monoids and then considering ascent to monoid algebras.

\medskip
\subsection{Two Classes of \texorpdfstring{Rank-$1$}{Rank-1} MCD Monoids}

First, we identify two classes of MCD rank-$1$ monoids; one of these classes consists of atomic monoids while the other consists of antimatter monoids (i.e., monoids having no atoms).

\smallskip
\subsubsection{A Class of Atomic MCD Monoids}

Gotti and Li proved that the prime reciprocal monoid associated with the full sequence of primes is strongly atomic~\cite[Example~3.9 and Proposition~3.10]{GL22}. Liang, Wang, and Zhong extended the construction to arbitrary increasing sequences of primes and established the MCD property~\cite[Theorem~5.3]{LWZ24}. We further extend their family by allowing each defining denominator to contain an arbitrary fixed number of primes. Fix a pair $(P,k)$, where $P$ is an infinite set of primes and $k \in \nn$, and let $(p_n)_{n \ge 1}$ be the strictly increasing sequence with underlying set $P$. Throughout the first part of this section, we will consider the following rank-$1$ torsion-free monoid:
\begin{equation} \label{eq:prime-reciprocal-monoid}
    M_{P,k} := \Big\langle \frac{1}{p_j p_{j+2} \dots p_{j+2k}} : j \in \nn \Big\rangle \subseteq \qq_{\ge 0}.
\end{equation}
To be consistent with~\cite{LWZ24}, we call $M_{P,k}$ the \emph{prime reciprocal monoid} induced by the pair~$(P,k)$. The set of atoms of $M_{P,k}$ is the defining generating set
\begin{equation} \label{eq:prime-reciprocal-atoms}
    \mathcal{A}(M_{P,k}) = \Big\{ a_{j,k} := \frac1{p_j p_{j+2} \cdots p_{j+2k}} : j \in \nn \Big\}.
\end{equation}
Indeed, fix $j \in \nn$. In a decomposition of $a_{j,k}$ as a sum of defining generators other than itself, every index would exceed $j$, because the sequence $(a_{i,k})_{i \ge 1}$ is strictly decreasing. The denominator of the resulting sum could not be divisible by $p_j$, whereas $p_j$ divides the denominator of $a_{j,k}$. Hence each defining generator is an atom, and $M_{P,k}$ is atomic.

On the other hand, $M_{P,k}$ does not satisfy the ACCP because $(1/p_{2n}+M_{P,k})_{n\ge1}$ is an ascending chain of principal ideals that does not stabilize. Indeed, for each $n \in \nn$, the identity
\[
    \frac1{p_{2n}} = \frac1{p_{2n+2}} + (p_{2n+2} - p_{2n})p_{2n+4} \cdots p_{2n+2k} \frac1{p_{2n} p_{2n+2} \dots p_{2n + 2k}}
\]
holds. Thus, $1/p_{2n+2}$ properly divides $1/p_{2n}$ in $M_{P,k}$, and the chain is strictly ascending. Here, and throughout, an empty product is $1$.

The following lemma extends the denominator reduction in~\cite[Lemma~5.2]{LWZ24} and will be used to prove both the MCD and MCD-finite properties of $M_{P,k}$. Since the defining generators of $M_{P,k}$ have squarefree denominators, every element of $M_{P,k}$ has a squarefree denominator whose prime factors belong to $P$.

\begin{lemma}\label{lem:denominator-reduction-MPk}
    Let $P$ be an infinite set of primes, let $k\in\nn$, and set $M:=M_{P,k}$.
    \begin{enumerate}
        \item[(1)] If $r\in M$ has largest denominator prime $p_\ell$ with $\ell>2k+2$ then $a_{\ell-2k,k}\mid_M r$.
        \smallskip

        \item[(2)] Let $N\ge 2k+2$, set $Q_N:=\prod_{j=1}^N p_j$, and let $T$ be a nonempty finite subset of $M$ such that $\mathsf d(t)\mid Q_N$ for every $t\in T$. If $T$ has a nonzero common divisor in $M$ then it has a nonzero common divisor $c$ satisfying $\mathsf d(c)\mid Q_N$.
    \end{enumerate}
\end{lemma}

\begin{proof}
    (1) Let $r$ be as in the statement of the lemma. Choose, among all representations of the form
    \[
        r := \sum_{j=1}^m b_ja_{j,k},
    \]
    where $b_1,\dots,b_m \in \nn_0$ and $b_m \ge 1$, one with~$m$ minimal. Since $p_\ell$ divides $\mathsf d(r)$, we see that $m+2k \ge \ell$. Suppose that $m + 2k > \ell$. Then $p_{m+2k}$ divides neither $\mathsf d(r)$ nor the denominator of any $a_{j,k}$ with $j<m$. Taking $p_{m+2k}$-adic valuations in the displayed equality gives $p_{m+2k} \mid b_m$. Writing $b_m=cp_{m+2k}$ for some $c\in\nn$ and using $m>\ell-2k>2$, we obtain
    \[
        b_ma_{m,k}=cp_{m+2k}a_{m,k}=cp_{m-2}a_{m-2,k}.
    \]
    This gives a representation of $r$ with a smaller largest index, contradicting the minimality of $m$. Hence $m+2k=\ell$, and therefore $a_{\ell-2k,k}\mid_M r$.
    \smallskip

    (2) Let $c$ be a nonzero common divisor of $T$. If $\mathsf d(c)\mid Q_N$ then there is nothing to prove. Otherwise, let $p_\ell$ be the largest prime dividing $\mathsf d(c)$; then $\ell > N$. Set $a := a_{\ell-2k, k}$. By the first assertion, $a\mid_M c$. Fix $t\in T$. Since $c\mid_M t$ and $p_\ell$ divides $\mathsf d(c)$ but not $\mathsf d(t)$, the $p_\ell$-adic valuation of $t-c$ equals that of $c$. Moreover, no prime with index greater than $\ell$ divides $\mathsf d(t-c)$. Hence $p_\ell$ is the largest prime dividing $\mathsf d(t-c)$. The first assertion therefore gives $a\mid_M t-c$, and so $2a\mid_M t$.

    We now argue inductively that $qa\mid_M t$ for every $q\in\ldb1,p_\ell\rdb$. Suppose that $qa\mid_M t$ for some $q<p_\ell$. Since $p_\ell\nmid q$ and $p_\ell$ does not divide $\mathsf d(t)$, the $p_\ell$-adic valuation of $t-qa$ equals $-1$, and $p_\ell$ is again the largest prime dividing its denominator. Applying the first assertion to $t-qa$ yields $a\mid_M t-qa$, and therefore $(q+1)a\mid_M t$. Thus,
    \[
        p_\ell a
        =\frac{1}{p_{\ell-2k}p_{\ell-2k+2}\cdots p_{\ell-2}}
    \]
    is a common divisor of $T$ whose largest denominator prime has index at most $\ell-2$. Repeating this reduction finitely many times produces a nonzero common divisor $c'$ of $T$ with $\mathsf d(c')\mid Q_N$.
\end{proof}

We can now prove directly that $M_{P,k}$ is an MCD monoid, generalizing~\cite[Theorem~5.3]{LWZ24}.

\begin{prop} \label{prop:MPk-is-MCD}
    For an infinite subset $P$ of $\pp$ and $k \in \nn$, let $M_{P,k}$ be the monoid defined in~\eqref{eq:prime-reciprocal-monoid}. Then $M_{P,k}$ is an MCD monoid.
\end{prop}

\begin{proof}
    Set $M:=M_{P,k}$, and let $S$ be a nonempty finite subset of $M$. Choose $N\ge2k+2$ such that $\mathsf d(s)\mid Q_N:=\prod_{j=1}^N p_j$ for every $s\in S$. Consider the set
    \[
        \mathcal C:=\big\{d\in M:d\mid_M s\text{ for every }s\in S
        \text{ and }\mathsf d(d)\mid Q_N\big\}.
    \]
    The set $\mathcal C$ is nonempty because $0\in\mathcal C$. It is also finite: every $d\in\mathcal C$ satisfies $0\le d\le\min S$ and can be written as $j/Q_N$ for some $j\in\nn_0$. Choose the largest element $d$ of $\mathcal C$ with respect to the usual order on $\qq$.

    Suppose that $d$ is not an MCD of $S$. Then the set $S-d:=\{s-d:s\in S\}$ has a nonzero common divisor. Since $\mathsf d(s-d)\mid Q_N$ for every $s\in S$, Lemma~\ref{lem:denominator-reduction-MPk}(2) yields a nonzero common divisor $c$ of $S-d$ such that $\mathsf d(c)\mid Q_N$. It follows that $d+c\in\mathcal C$, whereas $d+c>d$, contradicting the choice of $d$. Therefore $d$ is an MCD of $S$, and $M_{P,k}$ is an MCD monoid.
\end{proof}

\smallskip
\subsubsection{A Class of Antimatter MCD Monoids}

We now turn to exhibiting a class of antimatter MCD rank-$1$ monoids. For each $p \in \pp$, the set $V_p:=\nn_0[1/p]$ of $p$-adic nonnegative rationals is an additive monoid in which divisibility agrees with the usual order. For each nonempty subset $P$ of $\pp$, we set
\begin{equation} \label{eq:antimatter-monoid}
    M_P := \sum_{p \in P} V_p \subseteq \qq_{\ge 0}.
\end{equation}
The monoid $M_P$ is generated by the elements $1/p^n$, where $p \in P$ and $n \in \nn_0$. Each such element is the sum of $p$ copies of $1/p^{n+1}$, so none of these generators is an atom. Thus, $M_P$ is antimatter.

The class of monoids we explore in this section is $\{M_P : \varnothing \ne P \subseteq \pp \}$. As with the class exhibited earlier, the elements of each monoid $M_P$ have a natural sum decomposition, which is illustrated in the following proposition.

\begin{prop} \label{prop:canonical-decomposition}
    Let $P$ be a nonempty set of primes, and consider the monoid $M_P := \sum_{p \in P} V_p$. Each element $q \in M_P$ can be written uniquely as follows:
    \begin{equation} \label{eq:canonical-decomposition}
        q = c_0 + \sum_{(p,n) \in P \times \nn} c_{p,n} \frac1{p^n},
    \end{equation}
    where $c_0 \in \nn_0$ and the family $(c_{p,n})_{(p,n)\in P\times\nn}$ consists of nonnegative integer coefficients, only finitely many of which are nonzero, such that $0 \le c_{p,n} < p$ for every $(p,n)\in P\times\nn$.
\end{prop}

\begin{proof}
    When $P$ is finite, this is precisely~\cite[Proposition~5.14]{CGGP26}. Now let $P$ be arbitrary. Since every element of $M_P$ is a finite sum of elements from the monoids $V_p$, each $q\in M_P$ belongs to $M_{P_0}$ for some finite nonempty subset $P_0$ of $P$. Applying the finite case in $M_{P_0}$ gives a decomposition of the required form, with $c_{p,n}=0$ for $p\in P\setminus P_0$.

    For uniqueness, suppose that $q$ has two decompositions of the stated form. Only finitely many primes occur with a nonzero coefficient in either decomposition. Let $P_0$ be their union, enlarged by one element of $P$ if this union is empty. Both decompositions are then canonical sum decompositions of $q$ in $M_{P_0}$, so they coincide by the finite case. Hence the decomposition is unique in $M_P$.
\end{proof}

Let $c_0\colon M_P\to\nn_0$ and, for each $(p,n)\in P\times\nn$, let $c_{p,n}\colon M_P\to\nn_0$ be the coefficient functions defined by $c_0(q) = c_0$ and $c_{p,n}(q) = c_{p,n}$, where $c_0$ and $c_{p,n}$ are the coefficients in the canonical sum decomposition of $q$ in~\eqref{eq:canonical-decomposition}. Now we introduce the following terminology.

\begin{definition}
    Let $P$ be a nonempty set of primes, and let $M_P$ be the monoid in~\eqref{eq:antimatter-monoid}.
    \begin{enumerate}
        \item[(1)] For each $q \in M_P$, we call the right-hand side of~\eqref{eq:canonical-decomposition} the \emph{canonical sum decomposition} of $q$ in $M_P$.
        \smallskip

        \item[(2)] For each pair $(p,n) \in P \times \nn$, we call the function $c_{p,n} \colon M_P \to \nn_0$ the $(p,n)$-\emph{canonical projection} of $M_P$.
    \end{enumerate}
\end{definition}

For finite $P$, it was also proved in~\cite[Proposition~5.16]{CGGP26} that the canonical $(p,n)$-projections satisfy the following basic properties. They remain valid for arbitrary $P$: for any fixed elements under consideration, choose a finite subset $P_0$ of $P$ containing every prime occurring in their canonical sum decompositions, as well as the specified prime $p$ when needed, and apply the finite case in $M_{P_0}$. Proposition~\ref{prop:canonical-decomposition} ensures that the resulting canonical coefficients agree with those computed in $M_P$.
\begin{enumerate}
    \item[(1)] For each $p \in P$ and $q \in M_P$, the inequality $v_p(q) \ge 0$ holds if and only if $c_{p,n}(q) = 0$ for every $n \in \nn$.
    \smallskip

    \item[(2)] For each $(p,n) \in P \times \nn$, the equality $c_{p,n}(q+r) = c_{p,n}(q) + c_{p,n}(r)$ holds for all $q,r \in M_P$ such that $\gcd(\mathsf{d}(q),\mathsf{d}(r))=1$.
    \smallskip

    \item[(3)] $c_0(q+r) = c_0(q) + c_0(r)$ for all $q,r \in M_P$ such that $\gcd(\mathsf{d}(q), \mathsf{d}(r)) = 1$.
\end{enumerate}

Fix $q \in M_P$. It will be convenient in our proof that $M_P$ is an MCD monoid to compact the canonical sum decompositions of $q$ inside $M_P$ by setting, for each $p \in P$,
\[
    c_p(q) := \sum_{n \in \nn} c_{p,n}(q)\frac1{p^n} \in V_p \cap [0,1),
\]
and then rewriting the canonical sum decomposition of $q$ as follows:
\begin{equation}
    q = c_0(q) + \sum_{p \in P} c_p(q).
\end{equation}
As the following proposition indicates, canonical sum decompositions inside $M_P$ behave well with respect to divisibility.
\smallskip

\begin{prop} \label{prop:canonical-divisibility}
    Let $P$ be a nonempty set of primes, and let $M_P$ be the monoid defined in~\eqref{eq:antimatter-monoid}. For any $r,s \in M_P$, the following statements hold:
    \begin{enumerate}
        \item[(1)] $c_0(r)$ is the largest integer dividing $r$ in $M_P$.
        \smallskip

        \item[(2)] If $r \mid_{M_P} s$ then $c_0(r) \le c_0(s)$.
        \smallskip

        \item[(3)] If $r \mid_{M_P} s$ and $c_p(r) > c_p(s)$ for some $p \in P$ then $c_0(r) < c_0(s)$.
    \end{enumerate}
\end{prop}

\begin{proof}
    When $P$ is finite, these statements follow from~\cite[Proposition~5.17]{CGGP26}. We reduce the general case to that one.

    The canonical sum decomposition shows that $c_0(r)$ divides $r$ in $M_P$. If an integer $m$ divides $r$ in $M_P$, write $r=m+u$ for some $u\in M_P$. Choose a finite nonempty subset $P_0$ of $P$ containing every prime occurring in the canonical sum decompositions of $r$ and $u$. Then $m$ divides $r$ in $M_{P_0}$, and Proposition~\ref{prop:canonical-decomposition} shows that the integer coefficient of $r$ computed in $M_{P_0}$ is $c_0(r)$. The finite case therefore yields $m\le c_0(r)$, proving~(1).

    Now suppose that $r\mid_{M_P}s$, and write $s=r+u$ for some $u\in M_P$. Choose a finite nonempty subset $P_0$ of $P$ containing every prime occurring in the canonical sum decompositions of $r$, $s$, and $u$. Then $r\mid_{M_{P_0}}s$, so the finite case gives $c_0(r)\le c_0(s)$, proving~(2). Finally, if $c_p(r)>c_p(s)$ for some $p\in P$, enlarge $P_0$ to contain $p$. The finite case then gives $c_0(r)<c_0(s)$, proving~(3).
\end{proof}

We now prove that $M_P$ is an MCD monoid for every nonempty set $P$ of primes.

\begin{theorem} \label{thm:antimatter-mcd}
    For every nonempty set $P$ of primes, the monoid $M_P$ defined in~\eqref{eq:antimatter-monoid} is an MCD monoid.
\end{theorem}

\begin{proof}
    Fix a nonempty finite subset $S$ of $M_P$ and let us prove that $S$ has an MCD inside $M_P$. If $0 \in S$ then $0$ is the only common divisor of $S$ in $M_P$, and so $0$ is an MCD of $S$. Thus, from now on we can assume that $S \subseteq M_P^\bullet$. For each $s \in S$, write its compact canonical sum decomposition:
    \[
        s = c_0(s) + \sum_{p \in P} c_p(s),
    \]
    where $c_p(s) = \sum_{n \in \nn} c_{p,n}(s)/p^n$ for each $p \in P$. Now set
    \[
        C := \min\{c_0(s) : s \in S\} \quad \text{ and } \quad S_0 := \{s \in S : c_0(s) = C \}.
    \]
    Observe that $C$ is a common divisor of $S$, and part~(2) of Proposition~\ref{prop:canonical-divisibility} ensures that no common divisor of $S$ has integer coefficient larger than $C$ in its canonical decomposition, that is, $c_0(d) \le C$ for every common divisor of $S$ in $M_P$. Now consider the set
    \[
        \mathcal{D} := \{d \in M_P : d \text{ is a common divisor of } S \text{ and } c_0(d)=C\}.
    \]
    This set is nonempty because $C \in \mathcal{D}$. Although $\mathcal{D}$ need not be finite, we now show that every element of $\mathcal{D}$ is dominated by an element in a finite subset of $\mathcal{D}$. Let
    \[
        F := \{p \in P : c_{p,n}(s) \ne 0 \text{ for some } s \in S \text{ and } n \in \nn\}.
    \]
    Observe that the set $F$ is finite. For each $p \in F$, set
    \[
        N_p := \max\{n \in \nn : c_{p,n}(s) \ne 0 \text{ for some } s \in S\}.
    \]
    For $p \in P \setminus F$, set $N_p:=0$ and consider the set
    \[
        \mathcal{D}_0 := \{d \in \mathcal{D} : c_{p,n}(d)=0 \text{ for all } p \in P \text{ and all } n>N_p\}.
    \]
    The set $\mathcal{D}_0$ contains $C$ and is finite: if $d \in \mathcal{D}_0$ and $c_{p,n}(d) \ne 0$ then $p \in F$ and $1 \le n \le N_p$, while the canonical decomposition bounds the corresponding coefficient by
    \[
        1 \le c_{p,n}(d) \le p-1.
    \]
    In particular, every possible denominator of an element of $\mathcal{D}_0$ divides the integer $\prod_{p \in F} p^{N_p}$.

    Let us verify that every element of $\mathcal{D}$ divides an element of $\mathcal{D}_0$. Fix $d \in \mathcal{D}$, and choose $s_0 \in S_0$. By part~(3) of Proposition~\ref{prop:canonical-divisibility}, we have $c_p(d) \le c_p(s_0)$ for every $p \in P$, so no prime outside $F$ occurs in the canonical decomposition of $d$. Suppose that $c_{p,n}(d)>0$ for some $p \in F$ and some $n>N_p$, and choose $n$ maximal for this prime. For each $s \in S$, write $s=d+u_s$, where $u_s \in M_P$. Since $v_p(d)=-n<v_p(s)$, we have $v_p(u_s)=-n$. Therefore $c_{p,n}(u_s)>0$, and $u_s-1/p^n \in M_P$. It follows that $d+1/p^n$ is a common divisor of $S$. Moreover, the relations $d \mid_{M_P} d+1/p^n \mid_{M_P} s_0$ and part~(2) of Proposition~\ref{prop:canonical-divisibility} give
    \[
        C=c_0(d) \le c_0(d+1/p^n) \le c_0(s_0)=C.
    \]
    Hence $d+1/p^n \in \mathcal{D}$. Starting from the current value of $d$, this step can be repeated $p-c_{p,n}(d)$ times: before the last step, the $n$-th digit remains nonzero, and the last step replaces it by $0$ while producing carries only at smaller indices. Thus, the largest index exceeding $N_p$ strictly decreases after finitely many steps. Repeating the procedure for the finitely many primes in $F$ terminates at an element of $\mathcal{D}_0$ divisible by the original $d$.

    Choose $s_0 \in S_0$, and let $m$ be the largest element of $\mathcal{D}_0$ in the usual order on $\qq$. Suppose that a nonzero $e \in M_P$ divides every element of $S-m$. Then $m+e$ is a common divisor of $S$. The relations $m \mid_{M_P} m+e \mid_{M_P} s_0$ and part~(2) of Proposition~\ref{prop:canonical-divisibility} imply that $c_0(m+e)=C$, so $m+e \in \mathcal{D}$. By the preceding paragraph, it divides an element of $\mathcal{D}_0$ larger than $m$, a contradiction. Therefore $m$ is an MCD of $S$.
\end{proof}

\medskip
\subsection{Ascent of the MCD Property to Monoid Algebras}

We begin with the pre-Schreier case, which follows from classical GCD ascent. The interpolation property below is the monoid form of~\cite[Theorem~1.1]{mZ87}; its proof uses only multiplicative divisibility and applies to cancellative commutative monoids.

\begin{prop}\label{prop:schreier-interpolation}
    Let $M$ be a pre-Schreier monoid. If $d_1,\dots,d_k,s_1,\dots,s_\ell\in M$ satisfy $d_i\mid_M s_j$ for every $(i,j)\in\ldb1,k\rdb\times\ldb1,\ell\rdb$ then there exists $d\in M$ such that $d_i\mid_M d\mid_M s_j$ for every such pair $(i,j)$.
\end{prop}

\begin{corollary}\label{cor:pre-schreier-ascent}
    Let $F$ be a field, and let $M$ be a torsion-free pre-Schreier monoid. If $M$ is an MCD monoid then $M$ is a GCD monoid and $F[M]$ is a GCD domain. In particular, $F[M]$ is an MCD domain.
\end{corollary}

\begin{proof}
    Let $S$ be a nonempty finite subset of $M$, let $d$ be an MCD of $S$, and let $e$ be any common divisor of $S$. Proposition~\ref{prop:schreier-interpolation} gives a common divisor $c$ of $S$ divisible by both $d$ and $e$. Since $d$ is an MCD, the elements $c$ and $d$ are associates. Thus, $e\mid_M d$, proving that $d$ is a GCD of $S$. Therefore $M$ is a GCD monoid, and the classical ascent theorem~\cite[Theorem~14.5]{rG84} implies that $F[M]$ is a GCD domain.
\end{proof}

We next establish an ascent criterion that does not require $M$ to be pre-Schreier. Its hypothesis concerns the group algebra $F[\gp(M)]$.

\begin{theorem} \label{thm:MCD-ascent-ambient-ACCP}
    Let $F$ be a field, let $M$ be a commutative cancellative monoid, and suppose that the Grothendieck group of $M$ is linearly orderable. If $F[G]$ satisfies the ACCP then the following conditions are equivalent.
    \begin{enumerate}
        \item[(a)] $M$ is an MCD monoid.
        \smallskip

        \item[(b)] $F[M]$ is an MCD domain.
    \end{enumerate}
\end{theorem}

\begin{proof}
    Set $R=F[M]$ and $L=F[G]$. We first observe that every unit of $L$ has the form $\lambda x^g$, where $\lambda\in F^\times$ and $g\in G$. Indeed, for nonzero $p,q\in L$, the least and greatest exponents in $pq$ are, respectively, the sums of the least and greatest exponents in $p$ and $q$. Consequently, a product can be a monomial only when both factors are monomials. In particular,
    \[
        R^\times
        =\{\lambda x^u : \lambda\in F^\times,\ u\in\uu(M)\}.
    \]

    Suppose first that $M$ is an MCD monoid, and fix nonzero elements $f_1,\dots,f_n\in R$. Let $\mathcal C$ be the set of their common divisors in $R$. For each $c\in\mathcal C$, write $f_1 = cq_c$ for some $q_c \in R$. Since $L$ satisfies ACCP, the nonempty collection $\{q_cL:c\in\mathcal C\}$ has a maximal member. Choose $c\in\mathcal C$ such that $q_cL$ is maximal, and write $f_i = cq_i$ for every $i \in \ldb 1,n \rdb$, where each $q_i$ belongs to $R$. Now set
    \[
        A = \bigcup_{i=1}^n\supp(q_i)\subseteq M.
    \]
    Choose an MCD $m \in M$ of the finite nonempty subset $A$ of $M$. Then $x^m$ divides every $q_i$ in $R$, and therefore $D := cx^m$ is a common divisor of $f_1,\dots,f_n$ in $R$. We claim that actually~$D$ is an MCD.
    \smallskip

    \noindent \textsc{Claim.} $D$ is an MCD of the set $\{f_1, \dots,f_n \}$.
    \smallskip

    \noindent \textsc{Proof of Claim.} Suppose that $E$ is a common divisor of $f_1,\dots,f_n$ and that $D\mid_R E$. Write $E=Dr$ for some $r\in R$. Writing $f_1=Eq_E$ and using $c\mid_R E$, we obtain
    \[
        q_c = (E/c)q_E.
    \]
    It follows that $q_cL\subseteq q_EL$, and the maximality of $q_cL$ forces equality. Hence $E/c = x^m r$ is a unit of $L$, and so $r$ is a unit of~$L$. By the observation in the first paragraph, we may write $r = \lambda x^g$ for some $\lambda\in F^\times$ and $g \in M$. Since $E$ divides every $f_i$, the element $m+g$ is a common divisor in $M$ of all the exponents in $A$. Moreover, $m\mid_M m+g$. The maximality of $m$ therefore gives $g \in \uu(M)$. Thus, $r \in R^\times$, proving that $D$ is an MCD of $\{f_1,\dots,f_n \}$, which establishes our claim.
    \smallskip

    Conversely, suppose that $R$ is an MCD domain, and fix $a_1,\dots,a_n\in M$. Let $d$ be an MCD in $R$ of the monomials $x^{a_1},\dots,x^{a_n}$. Since a divisor of a monomial in $R$ must itself be a monomial, we can write $d=\lambda x^a$ for some $\lambda \in F^\times$ and $a\in M$. It follows immediately from the maximality of $d$ that $a$ is an MCD in $M$ of $a_1, \dots, a_n$. Hence $M$ is an MCD monoid.
\end{proof}

The hypothesis on the group algebra has a classical interpretation in terms of its exponent group. For a torsion-free abelian group $G$, the algebra $F[G]$ is a GCD domain~\cite[Proposition~2.1]{AA82}. Consequently, it satisfies the ACCP if and only if it is a UFD. By~\cite[Proposition~2.2]{AA82}, this occurs precisely when $G$ satisfies the ascending chain condition on cyclic subgroups, or equivalently, when every rank-$1$ subgroup of $G$ is cyclic. The next corollary gives a useful special case.

\begin{corollary} \label{cor:MCD-ascent-free-quotient-group}
    Let $F$ be a field, and let $M$ be a linearly orderable monoid whose quotient group is free abelian. Then $M$ is an MCD monoid if and only if $F[M]$ is an MCD domain. In particular, this equivalence holds for every submonoid of $\zz^r$.
\end{corollary}

\begin{proof}
    Set $G:=\gp(M)$. Since $G$ is free abelian, $F[G]$ is a Laurent polynomial ring over $F$ (possibly in infinitely many variables). Thus, $F[G]$ is a UFD and, in particular, satisfies the ACCP. The result now follows from Theorem~\ref{thm:MCD-ascent-ambient-ACCP}.
\end{proof}

As the following example illustrates, the ACCP hypothesis on $F[G]$ is not necessary.

\begin{example}
    Let $F$ be a field and consider the monoid algebra $F[\qq]$, which is a B\'ezout domain: indeed, every finite subset is contained in a Laurent polynomial PID $F[x^{\pm 1/n}]$, and the B\'ezout identity obtained there remains valid in $F[\qq]$. On the other hand, the monoid algebra $F[\qq]$ does not satisfy the ACCP. Indeed, for every $n \in \nn$, the identity $1 - x^{1/2^n} = \big( 1 - x^{1/2^{n+1}}\big) \big( 1 - x^{1/2^{n+1}}\big)$ holds in $F[\qq]$, which yields the inclusion $1 - x^{1/2^n} F[\qq] \subseteq 1 - x^{1/2^{n+1}}F[\qq]$. Hence $\big( (1 - x^{1/2^n}) F[\qq] \big)_{n \ge 1}$ is an ascending chain of principal ideals. The same does not stabilize as the corresponding sequence of degrees $(1/2^n)_{n \ge 1}$ is strictly decreasing. The ACCP hypothesis in Theorem~\ref{thm:MCD-ascent-ambient-ACCP} is therefore not necessary for ascent. \hfill $\blacksquare$
\end{example}

One might ask whether the MCD property ascends to monoid algebras over fields for every linearly orderable monoid. The following theorem gives a negative answer in every prime characteristic.

\begin{theorem} \label{thm:MCD-nonascent-positive-characteristic}
    For every prime $p$, there exists a reduced, linearly orderable rank-$3$ MCD monoid $M_p$ such that $\ff_p[M_p]$ is not an MCD domain.
\end{theorem}

\begin{proof}
    Fix a prime $p$, set $D_p:=\zz[1/p]$, and consider the additive monoid
    \[
        M_p := \big\{(a,b,c,d)\in(D_p)_{\ge 0}^4 : a+b = c+d \big\}.
    \]
    This monoid is reduced and cancellative, and the lexicographic order on $D_p^4$ restricts to a compatible linear order on $M_p$. Its rank is $3$: it lies in the rational hyperplane $a+b=c+d$ and contains the three linearly independent vectors $(1,0,1,0)$, $(1,0,0,1)$, and $(0,1,1,0)$.
    \smallskip

    We first prove that $M_p$ is an MCD monoid. Let $S$ be a nonempty finite subset of $M_p$, and, for each $i\in\ldb 1,4\rdb$, let $\alpha_i$ be the minimum of the $i$-th coordinates of the elements of $S$. For $m,w\in M_p$, the relation $m\mid_{M_p}w$ holds if and only if $m_i\le w_i$ for every $i$, since the coordinatewise difference then belongs to $M_p$. Set
    \[
        \tau:=\min\{\alpha_1+\alpha_2,\alpha_3+\alpha_4\},
    \]
    and define
    \[
        m_1:=\min\{\alpha_1,\tau\},\qquad m_2:=\tau-m_1,
        \qquad m_3:=\min\{\alpha_3,\tau\},\qquad m_4:=\tau-m_3.
    \]
    Then $m:=(m_1,m_2,m_3,m_4)\in M_p$ and $0\le m_i\le\alpha_i$ for every $i$, so $m$ is a common divisor of $S$. Moreover, either $(m_1,m_2)=(\alpha_1,\alpha_2)$ or $(m_3,m_4)=(\alpha_3,\alpha_4)$. If $z\in M_p$ is a common divisor of $S-m$ then $0\le z_i\le\alpha_i-m_i$ for every $i$. Thus, either $z_1=z_2=0$ or $z_3=z_4=0$. The equality $z_1+z_2=z_3+z_4$ forces $z=0$. Hence $m$ is an MCD of $S$, proving $M_p$ is an MCD monoid.
    \smallskip

    We next construct two elements of $R:=\ff_p[M_p]$ having no MCD. Let $s,t,u,v$ be algebraically independent over $\ff_p$. For each $n\in\nn_0$, set
    \[
        s_n:=s^{1/p^n},\quad t_n:=t^{1/p^n},\quad
        u_n:=u^{1/p^n},\quad v_n:=v^{1/p^n},
    \]
    As the unique factorization property ascends to polynomial extensions, $B_n := \ff_p[s_n,t_n,u_n,v_n]$ is a UFD for every $n \in \nn$. We identify $R$ with the subalgebra of $B:=\bigcup_{n\ge 0}B_n$ spanned by the monomials $s^at^bu^cv^d$ with $(a,b,c,d)\in M_p$. In particular, a polynomial in $B_n$ belongs to $R$ if and only if every monomial appearing in it has equal total degrees in $(s_n,t_n)$ and in $(u_n,v_n)$. Also, $R^\times=\ff_p^\times$, since the units of $B$ are precisely the nonzero constants.
    \smallskip

    Choose a homogeneous irreducible polynomial $Q(u,v)\in\ff_p[u,v]$ of degree $e>1$ with $\gcd(e,p)=1$. Such a choice can be made explicitly: if $p$ is odd, choose a nonsquare $\eta\in\ff_p^\times$ and take $Q(u,v):=u^2-\eta v^2$ and $e:=2$; if $p=2$, take $Q(u,v):=u^3+uv^2+v^3$ and $e:=3$. In the latter case, irreducibility follows because the cubic $Q(u,1)=u^3+u+1$ has no root in $\ff_2$. Now set
    \[
        f:=st^{e-1}Q(u,v)\qquad\text{and}\qquad g:=s^eQ(u,v).
    \]
    Both $f$ and $g$ belong to $R$. For each $n\in\nn_0$, let $Q_n := Q(u_n,v_n)$. The Frobenius identity gives
    \[
        f=s_n^{p^n}t_n^{(e-1)p^n}Q_n^{p^n}
        \qquad\text{and}\qquad
        g=s_n^{ep^n}Q_n^{p^n}.
    \]
    The elements $s_n,t_n,Q_n$ are pairwise nonassociate prime elements of the UFD $B_n$.
    \smallskip

    Let $h \in R$ be any common divisor of $f$ and $g$. Choose $n\in\nn_0$ large enough that $h,f/h,g/h \in B_n$. Unique factorization in $B_n$ yields
    \[
        h=\lambda s_n^aQ_n^b
        \qquad\text{for some }\lambda\in\ff_p^\times
        \text{ and }a,b\in\ldb 0,p^n\rdb.
    \]
    Every monomial of $h$ has total degree $a$ in $(s_n,t_n)$ and total degree $eb$ in $(u_n,v_n)$. Since $h\in R$, we must have $a=eb$. Thus, $eb\le p^n$, and in fact $eb<p^n$ because $e>1$ and $\gcd(e,p)=1$.
    \smallskip

    For $\rho=b/p^n\in(D_p)_{\ge 0}$, write
    \[
        g^\rho:=s_n^{eb}Q_n^b.
    \]
    This notation is independent of the chosen representation of $\rho$, since $Q_{n+1}^p=Q_n$. Conversely, if $0\le\rho=b/p^n<1/e$ then $g^\rho\in R$ and
    \[
        \frac{f}{g^\rho}=s_n^{p^n-eb}t_n^{(e-1)p^n}Q_n^{p^n-b},
        \qquad
        \frac{g}{g^\rho}=s_n^{e(p^n-b)}Q_n^{p^n-b}.
    \]
    Both quotients have nonnegative exponents and equal total degrees in the two pairs of variables, so they belong to $R$. Consequently, the common divisors of $f$ and $g$ in $R$ are the elements in the set
    \[
        \big\{\lambda g^\rho : \lambda \in \ff_p^\times,
        \ \rho \in D_p, \ 0 \le \rho < 1/e \big\}.
    \]
    Given any such $\rho$, the density of $D_p$ in $\rr$ allows us to choose $\rho'\in D_p$ with $\rho<\rho'<1/e$. Then $\lambda g^{\rho'}$ is another common divisor, and
    \[
        \frac{\lambda g^{\rho'}}{\lambda g^\rho}=g^{\rho'-\rho}\in R\setminus R^\times.
    \]
    Hence every common divisor of $f$ and $g$ is properly divisible by another common divisor. Therefore the set $\{f,g\}$ has no MCD in $R$, which allows us to conclude that $\ff_p[M_p]$ is not an MCD domain.
\end{proof}


The ambient-group criterion allows us ensure the ascent of the MCD from monoids $M$ not necessarily pre-Schreier to their monoid algebras over fields. The following example illustrates this.

\begin{example}
    Consider the numerical monoid $M := \langle 2,3 \rangle = \{0\} \cup \nn_{\ge 2}$ satisfies the ACCP, but but is not an MCD monoid because, for instance, one can readily verify that the set of MCDs of the subset $\{5,6\}$ of $M$ is the subset $\{2,3\}$ of $M$, which is not a singleton up to associates. The proof of Corollary~\ref{cor:pre-schreier-ascent} shows that every MCD in a pre-Schreier monoid is a GCD, so this $M$ is not pre-Schreier. Nevertheless, the Grothendieck group of~$\gp(M)$ is $\zz$, and Corollary~\ref{cor:MCD-ascent-free-quotient-group} applies. \hfill $\blacksquare$
\end{example}

Even when the MCD property does ascend, taking a GCD in the ambient group algebra and translating it into $F[M]$ need not produce a common divisor in $F[M]$. The following example sheds some light upon this observation.

\begin{example}
    Let $F$ be a field, and let $R$ be the monoid algebra of the numerical monoid $M := \langle 2,3 \rangle$ over the field $F$. Then let $L := F[x^{\pm1}]$ be the ring of Laurent polynomials over $F$. Now set
    \[
        q_1 := 1 - x + x^2 \quad \text{ and } \quad q_2 := 1 - x + 2x^2.
    \]
    Then from the fact that $q_2 - q_1 = x^2 \in L^\times$ we obtain that $\gamma := 1+x$ is a GCD in~$L$ of the Laurent polynomials
    \[
        f_1 := \gamma q_1 = 1+x^3
        \quad\text{ and }\quad
        f_2 := \gamma q_2 = 1+x^2+2x^3.
    \]
    However, no monomial translate of $\gamma$ is a common divisor of $f_1$ and $f_2$ in $R$. Indeed, if $x^t\gamma$ were such a divisor then $x^t\gamma\in R$ and $x^{-t}q_i\in R$ would force both $t$ and $-t$ to belong to the reduced monoid $M$. Hence $t=0$, whereas $1+x\notin R$. In this example $1$ is nevertheless an MCD of the set $\{f_1,f_2\}$ in~$R$. \hfill $\blacksquare$
\end{example}

Theorem~\ref{thm:MCD-nonascent-positive-characteristic} exhibits a pathological behavior that the MCD property of the exponent monoid cannot prevent: the common divisors can increase indefinitely through nonmonomial factors. Thus, the ACCP argument in Theorem~\ref{thm:MCD-ascent-ambient-ACCP} cannot be extended to arbitrary linearly orderable monoids without an additional hypothesis.

\medskip
\subsection{On the Ascent of Atomicity to Monoid Algebras over Fields}

Roitman proved that an atomic MCD domain has an atomic polynomial extension~\cite[Proposition~1.1]{mR93}, even though atomicity alone does not ascend to polynomial extensions. The following theorem shows that the corresponding statement fails for monoid algebras over fields: the exponent monoid can be both atomic and MCD while its monoid algebra is not atomic. The construction builds on~\cite{CG19,GGP25} and complements the non-ascent results in~\cite{GR25} by imposing the MCD property on the exponent monoid.

\begin{theorem} \label{thm:atomic-MCD-nonascent}
    For each $p \in \pp$, there exists an atomic rank-$1$ torsion-free MCD monoid $M$ whose monoid algebra $\ff_p[M]$ is not atomic.
\end{theorem}

\begin{proof}
    Fix $p \in \pp$. We first choose a prime $q \ne p$ and a nonconstant polynomial $P_q(x) \in \ff_p[x]$ of degree at most~$2$ such that $P_q(x^{q^n})$ is irreducible in $\ff_p[x]$ for every $n \in \nn$. The ingredients used in the proof of~\cite[Theorem~5.6]{GGP25} provide this choice. When $p = 2$, take $q = 3$ and $P_q(x) = x^2+x+1$. When $p \equiv 1 \pmod{4}$, take $q = 2$ and $P_q(x) = x-a$, where $a$ generates $\ff_p^\times$; the binomial criterion in~\cite[Theorem~5.3]{GGP25} applies to every exponent $2^n$. When $p \equiv 3 \pmod{4}$, take $q = 2$ and the quadratic polynomial supplied by~\cite[Theorem~5.5]{GGP25}. Set
    \[
        \varepsilon := v_q(2) \in \{0,1\},2er
    \]
    and use the convention $v_q(0) = \infty$. Set $\ell_0 := 0$, and let $(\ell_n)_{n \ge 1}$ be a strictly increasing sequence of positive integers such that $q^{\ell_n-2\ell_{n-1}} > 2p^{n+1}$ for every $n \in \nn$. For each $n \in \nn$, set
    \[
        a_n := \frac{p^n q^{\ell_n} - 1}{2p^{2n}q^{\ell_n}} \quad \text{ and } \quad b_n := \frac{p^n q^{\ell_n} + 1}{2p^{2n}q^{\ell_n}}.
    \]
    The construction in the proof of~\cite[Proposition~5.2]{GGP25} applies to this sequence, since our growth assumption implies $q^{\ell_n-\ell_{n-1}} > 2p^{n+1}$. It shows that the rank-$1$ torsion-free monoid $M_{p,q}$ generated by $A_{p,q} := \{a_n,b_n : n \in \nn\}$ is atomic with set of atoms $A_{p,q}$ and satisfies
    \[
        \nn_0\bigg[\frac1p\bigg] \subseteq M_{p,q} \subseteq \zz\bigg[\frac1p, \frac1q\bigg].
    \]
    Set $M := M_{p,q}$. We first argue that $M$ is an MCD monoid. For this, we prove the following claim.
    \smallskip

    \noindent \textsc{Claim 1.} For each $m \in M$, there exists $N \in \nn$ such that every representation
    \[
        m = \sum_{i = 1}^{L}(\alpha_i a_i+\beta_i b_i),
        \qquad
        \alpha_i,\beta_i \in \nn_0,
    \]
    satisfies $\alpha_i = \beta_i$ whenever $N < i \le L$. Consequently, for every $n > N$ and $h \in \nn_0$, the element $ha_n+hb_n$ divides $m$ in $M$ if and only if $ha_n$ or $hb_n$ divides $m$ in $M$.
    \smallskip

    \noindent \textsc{Proof of Claim 1.} Fix $m \in M$. Choose $N \in \nn$ such that $v_q(m) > -\ell_{N-1}-\varepsilon$ and $q^{\ell_{N-1}}-1 > m$. Consider a representation of $m$ as in the claim, with coefficients extended by zero beyond $L$. Suppose, by way of contradiction, that $j > N$ is the largest index for which $\alpha_j \ne \beta_j$. Since $v_q(a_i) = v_q(b_i) = -\ell_i-\varepsilon$, for each $i \in \ldb 1,j-1 \rdb$ we have
    \[
        v_q(\alpha_i a_i+\beta_i b_i) \ge -\ell_i-\varepsilon \ge -\ell_{j-1}-\varepsilon.
    \]
    Similarly, for each $i > j$, the maximality of $j$ gives $\alpha_i = \beta_i$, and so
    \[
        v_q(\alpha_i a_i+\beta_i b_i) = v_q\bigg(\frac{\alpha_i}{p^i}\bigg) \ge 0 \ge -\ell_{j-1}-\varepsilon.
    \]
    Since $v_q(m) > -\ell_{N-1}-\varepsilon \ge -\ell_{j-1}-\varepsilon$, we have
    \[
        v_q(\alpha_j a_j+\beta_j b_j) \ge -\ell_{j-1}-\varepsilon.
    \]
    On the other hand,
    \[
        \alpha_j a_j+\beta_j b_j
        = \frac{(\alpha_j+\beta_j)p^jq^{\ell_j}+\beta_j-\alpha_j}{2p^{2j}q^{\ell_j}}.
    \]
    Since the denominator has $q$-adic valuation $\ell_j+\varepsilon$, the preceding bound implies that the numerator is divisible by $q^{\ell_j-\ell_{j-1}}$. Its first summand is divisible by $q^{\ell_j}$, and therefore $\alpha_j-\beta_j$ is divisible in $\zz$ by $q^{\ell_j-\ell_{j-1}}$. Since $\alpha_j \ne \beta_j$, it follows that $\max\{\alpha_j,\beta_j\} \ge q^{\ell_j-\ell_{j-1}}$. Assume that $\alpha_j \ge q^{\ell_j-\ell_{j-1}}$, since the case involving $\beta_j$ is analogous. Then
    \[
        m \ge \alpha_j a_j \ge q^{\ell_j-\ell_{j-1}} \cdot \frac{p^jq^{\ell_j}-1}{2p^{2j}q^{\ell_j}} = \frac{p^jq^{\ell_j}-1}{2p^{2j}q^{\ell_{j-1}}} \ge q^{\ell_{j-1}}-1,
    \]
    using our bounds on the sequence $(\ell_n)_{n \ge 1}$. This contradicts $q^{\ell_{j-1}}-1 > m$. Therefore $\alpha_i = \beta_i$ for every $i > N$. If $ha_n$ or $hb_n$ divides $m$ with $n > N$ then a representation witnessing this divisibility has one of the corresponding coefficients at least $h$, and the other coefficient must have the same value. Thus, $ha_n+hb_n$ divides $m$. The converse follows because both $ha_n$ and $hb_n$ divide $ha_n+hb_n$. This proves the claim.
    \smallskip

    We next show that the divisors belonging to $\nn_0[1/p]$ have a maximum.
    \smallskip

    \noindent \textsc{Claim 2.} For each $m \in M$, the set of elements of $\nn_0[1/p]$ dividing $m$ in $M$ has a largest element in the usual order on $\qq$.
    \smallskip

    \noindent \textsc{Proof of Claim 2.} If $m \in \nn_0\big[1/p\big]$ then we can take $d := m$. Therefore we assume that $m \notin \nn_0\big[1/p\big]$. Take $N \in \nn$ such that $v_q(m) > -\ell_{N-1}-\varepsilon$ and $q^{\ell_{N-1}}-1 > m$. In light of Claim~1, we can write $m$ as
    \[
        m = u + \sum_{i = 1}^N \left(\alpha_i a_i + \beta_i b_i\right),
    \]
    where $\alpha_i,\beta_i \in \nn_0$ for every $i \in \ldb 1,N \rdb$ and $u \in \nn_0[1/p]$, because the coefficients at $a_n$ and $b_n$ agree for every $n > N$, and $a_n+b_n = 1/p^n$. Note that $u$ is uniquely determined by the $(2N)$-tuple $(\alpha_1,\beta_1,\dots,\alpha_N,\beta_N)$, and also that
    \[
        \alpha_i,\beta_i \le m \bigg( \frac{p^iq^{\ell_i}-1}{2p^{2i}q^{\ell_i}} \bigg)^{-1} < 2mp^{i+1}.
    \]
    This implies that there are finitely many $(2N+1)$-tuples $(\alpha_1,\beta_1,\dots,\alpha_N,\beta_N,u)$ where $\alpha_i,\beta_i \in \nn_0$ for $i \in \ldb 1,N \rdb$ and $u \in \nn_0\big[1/p\big]$ with
    \[
        m = u+\sum_{i = 1}^N \left(\alpha_i a_i + \beta_i b_i \right).
    \]
    Let $d$ be the maximum value of $u$ across all such tuples. Now take $d' \in \nn_0\big[1/p\big]$ with $d' > d$ and suppose, by way of contradiction, that $d' \mid_M m$. Since $m \notin \nn_0\big[1/p\big]$, we see that $v_q(m) < 0$, whence
    \[
        v_q(m-d') = \min\{v_q(m),v_q(d')\} = v_q(m),
    \]
    so $v_q(m-d') > -\ell_{N-1}-\varepsilon$. Also, $q^{\ell_{N-1}}-1 > m > m-d'$. Thus, by the same argument as above, we can write
    \[
        m-d' = u'+\sum_{i = 1}^N \left(\alpha_i'a_i + \beta_i'b_i\right),
    \]
    where $\alpha_i',\beta_i' \in \nn_0$ for every $i \in \ldb 1,N \rdb$ and $u' \in \nn_0\big[1/p\big]$. Therefore
    \[
        m = (d'+u')+\sum_{i = 1}^N \left(\alpha_i'a_i + \beta_i'b_i\right).
    \]
    However, as $d'+u' > d$ and $(\alpha_1',\beta_1',\dots,\alpha_N',\beta_N',d'+u')$ is a valid $(2N+1)$-tuple from above, this contradicts the maximality of~$d$. Hence no element of $\nn_0[1/p]$ greater than $d$ divides $m$, and Claim~2 is established.
    \smallskip

    We now prove that $M$ is an MCD monoid. Fix a nonempty finite subset $S = \{s_1,\dots,s_n\}$ of~$M$. If $0 \in S$ then $0$ is its only common divisor, so we may assume that every element of $S$ is positive. Take $N \in \nn$ such that $v_q(s) > -\ell_N-\varepsilon$ and $q^{\ell_N}-1 > s$ for every $s \in S$, and enlarge $N$, if necessary, so that Claim~1 applies to every $s \in S$ and every index greater than $N$. We construct sequences $c_1,c_2,\dots,c_N \in \nn_0$ and $d_1,d_2,\dots,d_N \in \nn_0$ satisfying the following conditions for each $i \in \ldb 1,N \rdb$:
    \begin{enumerate}
        \item[(1)] $c_i$ is the largest nonnegative integer such that $c_ia_i \mid_M s-\sum_{j = 1}^{i-1}(c_ja_j+d_jb_j)$ for all $s \in S$, and
        \smallskip

        \item[(2)] $d_i$ is the largest nonnegative integer such that $d_ib_i \mid_M s-\big(c_ia_i+\sum_{j = 1}^{i-1}(c_ja_j+d_jb_j)\big)$ for all $s \in S$.
    \end{enumerate}
    Now, for each $i \in \ldb 1,n \rdb$, define
    \[
        s_i' := s_i - \sum_{j = 1}^N \left(c_ja_j+d_jb_j\right).
    \]
    Moreover, for each $i \in \ldb 1,n \rdb$, let $r_i$ denote the largest element of $\nn_0\big[1/p\big]$ dividing $s_i'$, whose existence follows from Claim~2. Choose $i_0 \in \ldb 1,n \rdb$ such that $r_{i_0} = \min\{r_1,r_2,\dots,r_n\}$, and consider the element
    \[
        d := r_{i_0}+\sum_{j = 1}^N \left(c_ja_j+d_jb_j\right).
    \]
    For every $i$, we have $r_i-r_{i_0} \in \nn_0[1/p] \subseteq M$, so $r_{i_0} \mid_M s_i'$. Hence $d$ is a common divisor of $S$. We claim that $d$ is an MCD of~$S$. Since $M$ is atomic, it suffices to show that no atom divides each element of $S-d$. We distinguish four cases according to the index and type of a hypothetical common atom.
    \smallskip

    \textsc{Case 1:} $a_k$ is a common divisor of $S-d$ for some $k \le N$. Fix $s \in S$. As $a_k \mid_M s-d$, we have $a_k \mid_M s-d+m$ for every $m \in M$. Taking
    \[
        m = r_{i_0}+d_kb_k+\sum_{j = k+1}^N \left(c_ja_j+d_jb_j\right),
    \]
    we obtain
    \[
        a_k \mid_M s-\bigg( c_ka_k+\sum_{j = 1}^{k-1}\left(c_ja_j+d_jb_j\right)\bigg).
    \]
    Hence $(c_k+1)a_k \mid_M s-\sum_{j = 1}^{k-1}(c_ja_j+d_jb_j)$, contradicting the maximality of $c_k$.
    \smallskip

    \textsc{Case 2:} $b_k$ is a common divisor of $S-d$ for some $k \le N$. In this case, an argument similar to that used in Case~1 yields
    \[
        (d_k+1)b_k \mid_M s-\bigg(c_ka_k+\sum_{j = 1}^{k-1}(c_ja_j+d_jb_j)\bigg)
    \]
    for every $s \in S$, contradicting the maximality of~$d_k$.
    \smallskip

    \textsc{Case 3:} $a_k$ is a common divisor of $S-d$ for some $k > N$. Fix $s \in S$ and note that
    \[
        v_q(d) \ge \min\left\{v_q(r_{i_0}),v_q(c_1a_1+d_1b_1),\dots,v_q(c_Na_N+d_Nb_N)\right\} \ge -\ell_N-\varepsilon.
    \]
    Now,
    \[
        v_q(s-d) \ge \min\{v_q(s),v_q(d)\} \ge -\ell_N-\varepsilon,
    \]
    and $q^{\ell_N}-1 > s \ge s-d$. Repeating the argument in the proof of Claim~1 with cutoff $N$ now shows that if $a_k \mid_M s-d$ then $a_k+b_k \mid_M s-d$ as well. However, $a_k+b_k = 1/p^k \in \nn_0\big[1/p\big]$ and divides $s_{i_0}-d = s_{i_0}'-r_{i_0}$. Thus, $r_{i_0}+1/p^k \mid_M s_{i_0}'$, contradicting the maximality of $r_{i_0}$.
    \smallskip

    \textsc{Case 4:} $b_k$ is a common divisor of $S-d$ for some $k > N$. By similar reasoning, $r_{i_0}+1/p^k \mid_M s_{i_0}'$, again contradicting the maximality of $r_{i_0}$.
    \bigskip

    We have therefore proved that none of the atoms of $M$ is a common divisor of $S-d$ in $M$, so $d$ is an MCD of~$S$ in $M$.
    \smallskip

    It remains to prove that $\ff_p[M]$ is not atomic. For this, we first observe that $P_q(x^{1/p^k}) \in \ff_p[M]$ for all $k \in \nn_0$ because $\nn_0[1/p] \subseteq M$. Let $f$ be a factor of $P_q$ in $\ff_p[M]$, and write $P_q = fg$ for some $g \in \ff_p[M]$. Choose $k \in \nn$ sufficiently large to clear the denominators in the supports of $f$ and $g$. Then
    \[
        f\big(x^{(pq)^k}\big)g\big(x^{(pq)^k}\big) = P_q\big(x^{(pq)^k}\big) = P_q\big(x^{q^k}\big)^{p^k}
    \]
    in $\ff_p[x]$. Since $P_q\big(x^{q^k}\big)$ is irreducible in the UFD $\ff_p[x]$, we can write
    \[
        f\big(x^{(pq)^k}\big) = uP_q\big(x^{q^k}\big)^t
    \]
    for some $u \in \ff_p^\times$ and $t \in \nn_0$. Changing variables gives $f = uP_q\big(x^{1/p^k}\big)^t$. Thus, each nonunit factor of $P_q$ in $\ff_p[M]$ has the form $uP_q\big(x^{1/p^k}\big)^t$ for some $u \in \ff_p^\times$, $k \in \nn$, and $t \in \nn$.

    Now suppose, by way of contradiction, that $\ff_p[M]$ is atomic. Then $P_q = a_1 \cdots a_n$ for some irreducibles $a_1, \dots, a_n$ in $\ff_p[M]$. Since $a_1$ is a nonunit factor of $P_q$, we can pick $u \in \ff_p^\times$, $k \in \nn$, and $t \in \nn$ such that $a_1 = uP_q\big(x^{1/p^k}\big)^t$. As $a_1$ is irreducible, the equality $t = 1$ holds. However,
    \[
        a_1 = uP_q\big(x^{1/p^k}\big) = uP_q\big(x^{1/p^{k+1}}\big)^p,
    \]
    which contradicts the irreducibility of $a_1$ in $\ff_p[M]$. Hence $\ff_p[M]$ is not atomic.
\end{proof}

\bigskip
\section{The MCD-Finite Property}
\label{sec:mcd-finite}

We now study the finiteness of maximal common divisors, first in the prime reciprocal monoids and then in polynomial rings and monoid algebras.

\medskip
\subsection{A Class of \texorpdfstring{Rank-$1$}{Rank-1} MCD-Finite Monoids}

We begin this section by proving that the prime reciprocal monoids introduced in~\eqref{eq:prime-reciprocal-monoid} have the MCD-finite property. Recall that
\[
    M_{P,k} := \Big\langle \frac{1}{p_j p_{j+2} \dots p_{j+2k}} : j \in \nn \Big\rangle \subseteq \qq_{\ge 0},
\]
where $(p_n)_{n \ge 1}$ is the strictly increasing sequence of primes with underlying set $P$ and $k \in \nn$. We retain the notation $a_{j,k}:=1/(p_jp_{j+2}\cdots p_{j+2k})$ from~\eqref{eq:prime-reciprocal-atoms}. Using Lemma~\ref{lem:denominator-reduction-MPk}, we bound the denominators of the MCDs of any fixed nonempty finite subset of $M_{P,k}$.

\begin{prop} \label{prop:MPk-is-MCD-finite}
    For an infinite subset $P$ of $\pp$ and $k \in \nn$, let $M_{P,k}$ be the monoid defined in~\eqref{eq:prime-reciprocal-monoid}. Then $M_{P,k}$ is an MCD-finite monoid.
\end{prop}

\begin{proof}
    Set $M:=M_{P,k}$, and fix a nonempty finite subset $S$ of $M$. If $0\in S$ then $0$ is the only common divisor of $S$, and there is nothing to prove. Thus, we may assume that $S\subseteq M\setminus\{0\}$. Choose $N\ge 2k+2$ such that $\mathsf d(s)\mid Q_N:=\prod_{j=1}^N p_j$ for every $s\in S$.

    Let $d$ be an MCD of $S$ in $M$. Suppose that $\mathsf d(d)$ has a prime factor with index greater than $N$, and let $p_\ell$ be its largest prime factor. For each $s\in S$, the prime $p_\ell$ divides $\mathsf d(d)$ but not $\mathsf d(s)$, so it is also the largest prime dividing $\mathsf d(s-d)$. Since $s-d\in M$ and $\ell>N\ge 2k+2$, Lemma~\ref{lem:denominator-reduction-MPk}(1) gives $a_{\ell-2k,k}\mid_M s-d$ for every $s\in S$. This contradicts the maximality of $d$. Therefore $\mathsf d(d)\mid Q_N$.

    Since every common divisor $d$ of $S$ satisfies $0\le d\le\min S$, we conclude that
    \[
        \mcd_M(S)\subseteq
        \left\{\frac{j}{Q_N}:j\in\ldb 0,\fl{Q_N\min S}\rdb\right\}.
    \]
    The set on the right is finite. Thus, $S$ has only finitely many MCDs in $M$, and $M_{P,k}$ is an MCD-finite monoid.
\end{proof}

\medskip
\subsection{Ascent of the MCD-Finite Property to Polynomial Rings}

Eftekhari and Khorsandi~\cite[Theorem~2.1]{EK18} characterized IDF ascent by MCD-finiteness of the coefficient domain. The following theorem shows that MCD-finiteness itself ascends without an IDF hypothesis. Its proof uses the same principle of normalizing scalar multiples through a finite set of coefficients.

\begin{theorem}\label{thm:polynomial-mcd-finite}
    An integral domain $R$ is MCD-finite if and only if $R[x]$ is MCD-finite.
\end{theorem}

\begin{proof}
    Suppose first that $R$ is MCD-finite, and let $F$ be its field of fractions. Fix a nonempty finite subset $S=\{f_1,\dots,f_n\}$ of $R[x]\setminus\{0\}$. Since $F[x]$ is a UFD, its divisors of $f_1$ belong to only finitely many associate classes. It therefore suffices to show that each such class contains only finitely many MCDs of $S$ up to associates in $R[x]$.

    Fix a class containing an MCD $a\in R[x]$ of $S$, and write $f_i=ab_i$ with $b_i\in R[x]$. Let $B\subseteq R\setminus\{0\}$ be the finite nonempty set of nonzero coefficients of the polynomials $b_i$, and let $\alpha\in R\setminus\{0\}$ be the leading coefficient of $a$. Every MCD in this class has the form $ra$ for some $r\in F^\times$. Since $ra\in R[x]$, its leading coefficient $\alpha r$ belongs to $R\setminus\{0\}$. Moreover, $b_i/r\in R[x]$ for every $i$, so
    \[
        \frac{\alpha B}{\alpha r}=r^{-1}B\subseteq R\setminus\{0\}.
    \]
    The set $r^{-1}B$ has no nonunit common divisor in $R$: such a divisor would divide every $b_i/r$ and contradict the maximality of $ra$. Thus, $\alpha r$ is an MCD of the fixed finite set $\alpha B$ in $R$. If two such elements $\alpha r$ and $\alpha r'$ are associates in $R$ then $r/r'\in R^\times$ and $ra,r'a$ are associates in $R[x]$. Since $R$ is MCD-finite, the chosen class contains only finitely many MCDs up to associates in $R[x]$. Therefore $R[x]$ is MCD-finite.

    Conversely, every divisor in $R[x]$ of a nonzero constant is constant, and $R[x]^\times=R^\times$. Thus, for every nonempty finite subset $S$ of $R\setminus\{0\}$, its MCDs and associate classes are the same in $R$ and $R[x]$. MCD-finiteness therefore descends from $R[x]$ to $R$.
\end{proof}

Iterating Theorem~\ref{thm:polynomial-mcd-finite} gives the same equivalence for a polynomial ring in any finite number of indeterminates.

\medskip
\subsection{The q-GCD Property}

We conclude by considering the q-GCD property, which is stronger than the MCD-finite property. Recall that a cancellative commutative monoid is a \emph{q-GCD monoid} if every nonempty finite subset has at most one MCD up to associates. An integral domain is a \emph{q-GCD domain} if its multiplicative monoid of nonzero elements is a q-GCD monoid. We prove that this property does not ascend to monoid algebras over fields, even for rank-$1$ torsion-free monoids. The following lemma supplies the exponent monoids for our construction.

\begin{lemma} \label{lem:common-divisors-qgcd}
    Let $M$ be an additive submonoid of $\qq_{\ge 0}$. If every nonempty finite subset of $M^\bullet$ has a common divisor in $M^\bullet$ then $M$ is a q-GCD monoid.
\end{lemma}

\begin{proof}
    Fix a nonempty finite subset $S$ of $M$. If $0 \in S$ then $0$ is its only common divisor. Suppose that $0 \notin S$ and that $d$ is an MCD of $S$. If $d \notin S$ then $S-d \subseteq M^\bullet$, so the hypothesis gives a nonzero common divisor of $S-d$, contradicting the maximality of $d$. Hence $d \in S$. Since $d \le s$ for every $s \in S$, we conclude that $d = \min S$. Thus, $S$ has at most one MCD, and $M$ is a q-GCD monoid.
\end{proof}

We now use this lemma to construct the desired counterexamples.

\begin{theorem} \label{thm:q-GCD-nonascent}
    For every $p \in \pp$, there exists a rank-$1$ cancellative torsion-free q-GCD monoid $M$ whose monoid algebra $\ff_p[M]$ is not a q-GCD domain.
\end{theorem}

\begin{proof}
    Fix $p \in \pp$, and choose an odd prime $q \ne p$. Set
    \[
        V := \zz[1/p]_{\ge 0},
        \qquad
        G := \zz[1/p,1/q],
        \qquad
        W := G_{\ge 0},
        \qquad
        M := V \cup G_{\ge 1}.
    \]
    Then $M$ is an additive submonoid of $\qq_{\ge 0}$ and is cancellative, torsion-free, and of rank~$1$. Every $m \in M^\bullet$ is divisible in $M$ by $1/p^n$ for all sufficiently large $n$. Indeed, if $m \in V$ then $m-1/p^n \in V$ for all sufficiently large $n$. Otherwise, $m > 1$, and $m-1/p^n \in G_{> 1}$ for all sufficiently large $n$. Consequently, every nonempty finite subset of $M^\bullet$ has a nonzero common divisor. Lemma~\ref{lem:common-divisors-qgcd} shows that $M$ is a q-GCD monoid.

    Set $R := \ff_p[M]$ and $L := \ff_p[W]$. The ring $L$ is a B\'ezout domain, since it is the increasing union of the principal ideal domains $\ff_p[x^{1/(pq)^n}]$ for $n \in \nn_0$. We will also use the fact that divisibility of ordinary polynomials descends from $L$ to $\ff_p[x]$. To see this, observe that $L$ is integral over $\ff_p[x]$, since each of its monomials is integral. If $A,B \in \ff_p[x]\setminus\{0\}$ and $B/A \in L$ then $B/A$ is an element of $\ff_p(x)$ integral over the integrally closed domain $\ff_p[x]$. Therefore $B/A \in \ff_p[x]$.

    In $L$, set
    \[
        A := 1+x^{1/(pq)}
        \quad \text{and} \quad
        B := 1+x^{1/(pq)}+x.
    \]
    We first show that every divisor in $L$ of either $A$ or $B$ that belongs to $\ff_p[V]$ is constant. Let $d \in \ff_p[V]$ be such a divisor, and choose $j \in \nn_0$ such that $d(x^{p^j}) \in \ff_p[x]$. Under the substitution $x \mapsto x^{p^{j+1}q}$, the divisibility relation remains valid in $L$ and hence descends to $\ff_p[x]$. Thus,
    \[
        d(x^{p^{j+1}q}) \mid_{\ff_p[x]} 1+x^{p^j}
        \quad \text{or} \quad
        d(x^{p^{j+1}q}) \mid_{\ff_p[x]} 1+x^{p^j}+x^{p^{j+1}q},
    \]
    respectively. In either case, $d(x^{p^{j+1}q})$ belongs to $\ff_p[x^q]$. Suppose that it is nonconstant, and choose one of its roots $\alpha$ in an algebraic closure of $\ff_p$. The polynomial on the corresponding right-hand side has constant coefficient $1$, so $\alpha \ne 0$. Choose a primitive $q$-th root of unity $\zeta$, which exists because $q \ne p$. Since $d(x^{p^{j+1}q}) \in \ff_p[x^q]$, the element $\alpha\zeta$ is also a root of $d(x^{p^{j+1}q})$. Subtracting the values at $\alpha\zeta$ and $\alpha$ of the polynomial on the corresponding right-hand side gives
    \[
        \alpha^{p^j}(\zeta^{p^j}-1) = 0.
    \]
    This is impossible because $\alpha \ne 0$ and $q \nmid p^j$. Hence $d$ is constant.

    Choose $k \in \nn$ such that $p^{k-1} > q$, and consider
    \[
        s_1 := xA^{p^k+1}B,
        \qquad
        s_2 := x^2A^2B,
        \qquad
        g := xAB,
        \qquad
        h := xA^2.
    \]
    Every exponent in the support of these four elements is at least $1$, so they all belong to $R$. The elements $g$ and $h$ are not associates, since $R^\times = \ff_p^\times$ and
    \[
        \deg g = 2+\frac{1}{pq} > 1+\frac{2}{pq} = \deg h.
    \]
    We will prove that $g$ and $h$ are MCDs of $S := \{s_1,s_2\}$ in $R$.

    First,
    \[
        \frac{s_1}{g} = A^{p^k} = 1+x^{p^{k-1}/q} \in R
        \quad \text{and} \quad
        \frac{s_2}{g} = xA \in R,
    \]
    so $g$ is a common divisor of $S$. The elements $x$ and $A$ are relatively prime already in $\ff_p[x^{1/(pq)}]$, and therefore they generate the unit ideal in $L$. It follows that
    \[
        A \in \gcd\nolimits_L(A^{p^k},xA).
    \]
    Let $d \in R$ be a common divisor of $S/g$. Then $d \mid_L A$, so $\deg d \le 1/(pq) < 1$. Since $M \cap [0,1] \subseteq V$, the element $d$ belongs to $\ff_p[V]$. The preceding argument shows that $d$ is a nonzero constant and hence a unit of $R$. Thus, $g$ is an MCD of $S$.

    Similarly,
    \[
        \frac{s_1}{h} = A^{p^k-1}B
        = 1+x^{p^{k-1}/q}+xA^{p^k-1} \in R
        \quad \text{and} \quad
        \frac{s_2}{h} = xB \in R,
    \]
    so $h$ is a common divisor of $S$. Since $x$ and $A$ generate the unit ideal in $L$, we have
    \[
        B \in \gcd\nolimits_L(A^{p^k-1}B,xB).
    \]
    Every common divisor $d \in R$ of $S/h$ therefore divides $B$ in $L$ and satisfies $\deg d \le \deg B = 1$. Thus, $d \in \ff_p[V]$, and the preceding argument again shows that $d$ is a nonzero constant. Hence $h$ is also an MCD of $S$. Since $g$ and $h$ are not associates, $R$ is not a q-GCD domain.
\end{proof}

\bigskip
\section*{Acknowledgments}

During the preparation of this paper, the authors were part of PRIMES-USA at MIT, and they would like to thank the program for making possible the rewarding research experience that led to this paper. The second author gratefully acknowledges support from the NSF under award DMS-2213323.

\bigskip

\end{document}